\documentclass{amsart}
\usepackage{amsfonts}

\usepackage{graphicx}
\usepackage{pstricks}
\usepackage{amscd}
\usepackage{amsmath}
\usepackage{amsxtra}

\newcommand{\strutstretchdef}{\newcommand{\strutstretch}{\vphantom{\raisebox{1pt}{$\big($}\raisebox{-1pt}{$\big($}}}}
\theoremstyle{plain}
\newtheorem{theorem}{Theorem}[section]
\newtheorem{lemma}[theorem]{Lemma}
\newtheorem{proposition}[theorem]{Proposition}

\theoremstyle{definition}

\theoremstyle{remark}

\numberwithin{equation}{section}

\newlength{\struh}
\newlength{\textminustop}
\strutstretchdef
\newcommand{\rank}{\text{rank\ }}

\begin{document}
\title[Concrete Solution to the Nonsingular Quartic Binary Moment Problem]{Concrete Solution to the Nonsingular \\ Quartic Binary Moment Problem}
\author{Ra\'{u}l E. Curto}
\address{Department of Mathematics, The University of Iowa, Iowa City, Iowa 52242}
\email{raul-curto@uiowa.edu}
\author{Seonguk Yoo}
\address{Department of Mathematics, Seoul National University, Seoul 151-742, Korea}
\email{seyoo73@gmail.com}
\thanks{The first named author was supported by NSF Grants DMS-0801168 and DMS-1302666.}

\begin{abstract}
Given real numbers $\beta \equiv \beta ^{\left( 4\right) }\colon \beta
_{00}$, $\beta _{10}$, $\beta _{01}$, $\beta _{20}$, $\beta _{11}$, $%
\beta _{02}$, $\beta _{30}$, $\beta _{21}$, $\beta _{12}$, $\beta _{03}$%
, $\beta _{40}$, $\beta _{31}$, $\beta _{22}$, $\beta _{13}$, $\beta
_{04}$, with $\beta _{00} >0$, the \emph{quartic real
moment problem} for $\beta $ entails finding conditions for the existence
of a positive Borel measure $\mu $, supported in $\mathbb{R}^2$, such that $\beta _{ij}=\int s^{i}t^{j}\,d\mu \;\;(0\leq i+j\leq
4) $. \ Let $\mathcal{M}(2)$ be the $6 \times 6$ moment matrix for $\beta^{(4)}$, given by $\mathcal{M}(2)_{\mathbf{i},\mathbf{j}}:=\beta_{\mathbf{i}+\mathbf{j}}$, where $\mathbf{i},\mathbf{j} \in \mathbb{Z}^2_+$ and $\left|\mathbf{i}\right|,\left|\mathbf{j}\right|\le 2$. \ In this note we find concrete representing measures for $\beta^{(4)}$ when $\mathcal{M}(2)$ is nonsingular; moreover, we prove that it is possible to ensure that one such representing measure is $6$-atomic.
\end{abstract}

\maketitle

\section{\label{Int}Introduction}

In this paper we find a direct proof that the \textit{nonsingular} Quartic Binary Moment Problem always admits a finitely atomic representing measure with the minimum number of atoms, that is, six atoms. \ We do this in three steps: 

\medskip
\noindent (i) by normalizing the given moment matrix $\mathcal{M}(2)$ to ensure that $\mathcal{M}(1)$ is the identity matrix; 

\medskip
\noindent (ii) by developing a new rank-reduction tool, which allows us to decompose the normalized $\mathcal{M}(2)$ matrix as the sum of a positive semidefinite moment matrix $\widetilde{\mathcal{M}(2)}$ of rank $5$ and the rank-one moment matrix of the point mass at the origin; and 

\medskip
\noindent (iii) by proving that when a moment matrix $\mathcal{M}(2)$ admits such a decomposition, and $\widetilde{\mathcal{M}(2)}$ admits a column relation subordinate to a degenerate hyperbola (i.e., a pair of intersecting lines), then $\mathcal{M}(2)$ admits a $6$-atomic representing measure (as opposed to the expected $7$-atomic measure).  

To describe our results in detail, we need some notation and terminology. \ Given real numbers $\beta \equiv \beta ^{\left( 4\right) }\colon \beta
_{00}$, $\beta _{10}$, $\beta _{01}$, $\beta _{20}$, $\beta _{11}$, $%
\beta _{02}$, $\beta _{30}$, $\beta _{21}$, $\beta _{12}$, $\beta _{03}$%
, $\beta _{40}$, $\beta _{31}$, $\beta _{22}$, $\beta _{13}$, $\beta
_{04}$, with $\beta _{00} >0$, the \emph{Quartic Real Moment Problem} for $\beta $ entails finding conditions for the existence
of a positive Borel measure $\mu $, supported in $\mathbb{R}^2$, such that $\beta _{ij}=\int s^{i}t^{j}\,d\mu \;\;(0\leq i+j\leq
4) $. \ Let $\mathcal{M}(2)$ be the moment matrix for $\beta^{(4)}$, given by $\mathcal{M}(2)_{\mathbf{i},\mathbf{j}}:=\beta_{\mathbf{i}+\mathbf{j}}$, where $\mathbf{i},\mathbf{j} \in \mathbb{Z}^2_+$ and $\left|\mathbf{i}\right|,\left|\mathbf{j}\right|\le 2$; this $6 \times 6$ matrix is shown below. \ (As is customary, the columns of $\mathcal{M}(2)$ are labeled $\textit{1},X,Y,X^2,XY,Y^2$. \ In a similar way, given a collection of real numbers $\beta^{(2n)}$ one defines the associated moment matrix by $\mathcal{M}(n)_{\mathbf{i},\mathbf{j}}:=\beta_{\mathbf{i}+\mathbf{j}}$, where $\mathbf{i},\mathbf{j} \in \mathbb{Z}^2_+$ and $\left|\mathbf{i}\right|,\left|\mathbf{j}\right|\le n$.) 
$$
\mathcal{M}(2)\equiv \begin{pmatrix}
\beta_{00} &\beta_{10} & \beta_{01} & \beta_{20} & \beta_{11} & \beta_{02} \\
\beta_{10} & \beta_{20} & \beta_{11} & \beta_{30} & \beta_{21} & \beta_{12} \\
\beta_{01} &\beta_{11} & \beta_{02} & \beta_{21} & \beta_{12} & \beta_{03} \\
\beta_{20} & \beta_{30} & \beta_{21} & \beta_{40} & \beta_{31} & \beta_{22} \\
\beta_{11} & \beta_{21} & \beta_{12} & \beta_{31} & \beta_{22} & \beta_{13} \\
\beta_{02} & \beta_{12} & \beta_{03} & \beta_{22} & \beta_{13} & \beta_{04}
\end{pmatrix}.
$$
Assume now that $\mathcal{M}(2)$ is nonsingular. \ A straightforward consequence of Hilbert's Theorem yields the existence of a finitely atomic representing measure, as follows. \ Let $\mathcal{P}_4$ be the cone of nonnegative polynomials of degree at most $4$ in $x,y$, regarded as a subset of $\mathbb{R}^{15}$. \ The dual cone is $\mathcal{P}_4^*:=\{\xi \in \mathbb{R}^{15}: \left\langle \xi,p \right\rangle \ge 0 \; \textrm{for all } p \in \mathcal{P}_4 \}$. \ If $a,b \in \mathbb{R}$ and $\xi_{(a,b)}:=(1,a,b,a^2,ab,b^2,a^3,a^2b,ab^2,b^3,a^4,a^3b,a^2b^2,ab^3,b^4) \in \mathbb{R}^{15}$, then $\left\langle \xi_{(a,b)},p\right\rangle=p(a,b) \ge 0$, for all $p \in \mathcal{P}_4$. \ Thus, $\xi_{(a,b)} \in \mathcal{P}_4^*$ for all $a,b \in \mathbb{R}$, and $\xi_{(a,b)}$ is also an extreme point. \ Consider now an arbitrary moment sequence $\beta^{(4)}$ with a nonsingular moment matrix $\mathcal{M}(2)$. \ Regarded as a point in $\mathbb{R}^{15}$, $\beta^{(4)}$ is in the interior of $\mathcal{P}_4^*$, since every $p \in \mathcal{P}_4$ is a sum of squares of polynomials. \ By the Krein-Milman Theorem and Carath\'eodory's Theorem, the Riesz functional $\Lambda_{\beta^{(4)}}$ is a convex combination of evaluations $\xi_{(a,b)}$; that is, $\beta^{(4)}$ admits a finitely atomic representing measure, with at most $15$ atoms. \ (In recent related work, L.A. Fialkow and J. Nie \cite{FiNi} have obtained this result as a consequence of a more general result on moment problems.) \ 

In this note we obtain a concrete $6$-atomic representing measure for $\mathcal{M}(2)$. \ The Quartic Real Binary Moment Problem admits an equivalent formulation in terms of complex numbers and representing measures supported in the complex plane $\mathbb{C}$, as follows.  Given complex numbers $\gamma \equiv \gamma ^{\left( 4\right) }\colon \gamma
_{00}$, $\gamma _{01}$, $\gamma _{10}$, $\gamma _{02}$, $\gamma _{11}$, $%
\gamma _{20}$, $\gamma _{03}$, $\gamma _{12}$, $\gamma _{21}$, $\gamma _{30}$%
, $\gamma _{04}$, $\gamma _{13}$, $\gamma _{22}$, $\gamma _{31}$, $\gamma
_{40}$, with $\gamma _{ij}=\bar{\gamma}_{ji}$, one seeks necessary and sufficient conditions for the existence
of a positive Borel measure $\mu $, supported in $\mathbb{C}$, such that 
\begin{equation*}
\gamma _{ij}=\int \bar{z}^{i}z^{j}\,d\mu \qquad (0\leq i+j\leq 4).
\end{equation*}
Just as in the real case, the Quartic Complex Moment Problem has an associated moment matrix $M(2)$, whose columns are conveniently labeled $1,Z,\bar{Z},Z^2,\bar{Z}Z,\bar{Z}^2$. \ The most interesting case of the Singular Quartic Binary Moment Problem arises when the rank of $M(2)$ is $5$, and the sixth column of $M(2)$, labeled $\bar{Z}^2$, is a linear combination of the remaining five columns. \ Depending on the coefficients in the linear combination, four subcases arise in terms of the associated conic $C$ \cite[Section 5]{tcmp6}: (i) $C$ is a parabola; (ii) $C$ is a nondegenerate hyperbola; (iii) $C$ is a pair of intersecting lines; and (iv) $C$ is a circle. \ In subcase (iii), it is possible to prove that the number of atoms in a representing measure (if it exists) may be $6$ \cite[Proposition 5.5 and Example 5.6]{tcmp6}; that is, in some soluble cases the rank of $M(2)$ may be strictly smaller than the number of atoms in any representing measure.

\begin{proposition} \label{degenerate} (\cite[Proposition 5.5]{tcmp6}) \ If $\mathcal{M}(2) \ge0$, if $\operatorname{rank} \mathcal{M}(2) =5$, and if  $XY = 0$ in the column space of $\mathcal{M}(2)$, then $\mathcal{M}(2)$ admits a representing measure $\mu$ with $\operatorname{card} \; \operatorname{supp} \mu \le 6$.   
\end{proposition}

When combined with previous work on truncated moment problems, Proposition \ref{degenerate} led to the following solution to the truncated moment problem on planar curves of degree $\leq 2$. \ Given a moment matrix $\mathcal{M}(n)$ and a polynomial $p(x,y)\equiv \sum p_{ij} x^i y^j$, we let $p(X,Y):=\sum p_{ij} X^i Y^j$. \ A column relation in $\mathcal{M}(n)$ is therefore always described as $p(X,Y)=0$ for some polynomial $p$, with deg $p \le n$. \ We say that $\mathcal{M}(n)$ is \textit{recursively generated} if for every $p$ with $p(X,Y)=0$ and every $q$ such that deg $pq \le n$ one has $(pq)(X,Y)=0$. \ In what follows, $v$ denotes the cardinality of the associated algebraic variety, defined as the intersection of the zero sets of all polynomials which describe the column relations in $\mathcal{M}(n)$.  

\begin{theorem}
\label{quartic} (\cite[Theorem 2.1]{tcmp9}, \cite[Theorem 1.2]{Fia4}) \ Let $p \in \mathbb{R}[x,y]$, with $\deg p(x,y)\leq 2$. \ Then $\beta ^{(2n)}$ has a representing measure supported
in the curve $p(x,y)=0$ if and only if $\mathcal{M}(n)$ has a column dependence relation $p(X,Y)=0$, $\mathcal{M}(n) \geq 0$, $\mathcal{M}(n)$ is recursively generated, and $r\leq v$.
\end{theorem}

The proof of Theorem \ref{quartic} made use of affine planar transformations to reduce a generic quadratic column relation to one of four canonical types: $Y=X^2$, $XY=1$, $XY=0$ and $X^2+Y^2=1$; each of these cases required an independent result. \ We shall have occasion to use the affine planar transformation approach in Section \ref{Tool}. \ To date, most of the existing theory of truncated moment problems is founded on the presence of nontrivial column relations in the moment matrix $\mathcal{M}(n)$. \ On one hand, when all columns labeled by monomials of degree $n$ can be expressed as linear combinations of columns labeled by monomials of lower degree, the matrix $\mathcal{M}(n)$ is flat, and the moment problem has a unique representing measure, which is finitely atomic, with exactly $\operatorname{rank} \mathcal{M}(n-1)$ atoms \cite[Theorem 1.1]{tcmp2}. \ As a straightforward consequence, we conclude that for $n=1$, an invertible $\mathcal{M}(n)$ always admits a flat extension, while that is not the case for $n \ge 3$; that is, there exist examples of positive and invertible $\mathcal{M}(3)$ without a representing measure (cf. \cite[Section 4]{tcmp2}).

When $n=2$, the idea is to extend the $6 \times 6$ moment matrix $\mathcal{M}(2)$ to a bigger $10 \times 10$ moment matrix $\mathcal{M}(3)$ by adding so-called $B$ and $C$ blocks, as follows:
\begin{equation*}
\mathcal{M}\left( 3\right) \equiv 
\begin{pmatrix}
\mathcal{M}\left( 2\right) & B\left( 3\right) \\ 
B\left( 3\right) ^{\ast} & C\left( 3\right)
\end{pmatrix}.
\end{equation*}
A result of J.L. Smul'jan \cite{Smu} states that $\mathcal{M}(3) \ge 0$ if and only if (i) $\mathcal{M}(2) \ge 0$; (ii) $B(3)=\mathcal{M}(2)W$ for some $W$; and (iii) $C(3) \ge W^{*}\mathcal{M}(2)W$. \ Moreover, $\mathcal{M}(3)$ is a \textit{flat} extension of $\mathcal{M}(2)$ (i.e., $\operatorname{rank} \mathcal{M}(3) = \operatorname{rank} \mathcal{M}(2)$)  if and only if $C(3)=W^{*}\mathcal{M}(2)W$. \ Further, when $\mathcal{M}(2)$ is invertible, one easily obtains $W=\mathcal{M}(2)^{-1}B(3)$, so in the flat extension case $C(3)$ can be written as $B(3)^{*}\mathcal{M}(2)^{-1}B(3)$. \ However, writing a general formula for $\mathcal{M}(3)$ is nontrivial, even with the aid of \textit{Mathematica}, because of the complexity of $(\mathcal{M}(2))^{-1}$ and the new moments contributed by the block $B(3)$. \ On the other hand, if only one column relation is present (given by $p(X,Y)=0$), then $v=+\infty$, and the condition $r \le v$, while necessary, will not suffice. \ One knows that the support of a representing measure must lie in the zero set of $p$, but this does not provide enough information to decipher the block $B(3)$. \ The situation is much more intriguing when no column relations are present; this is the nonsingular case, for which very little is known. 

\section{\label{Statement} Statement of the Main Result}

\begin{theorem} \label{MainTheorem}
Assume $\mathcal{M}(2)$ is positive and invertible. \ Then $\mathcal{M}(2)$ admits a representing measure, with exactly $6$ atoms; that is, $\mathcal{M}(2)$ actually admits a flat extension $\mathcal{M}(3)$.
\end{theorem}

The proof of Theorem \ref{MainTheorem} is constructive, in that we first prove that it is always possible to switch from the invertible $\mathcal{M}(2)$ to a related singular matrix $\widetilde{\mathcal{M}(2)}$, with $\operatorname{rank} \widetilde{\mathcal{M}(2)}=5$, for which Theorem \ref{quartic} applies. \ Since singular positive semidefinite matrices $\mathcal{M}(2)$ always admit representing measures with $6$ atoms or less, we can then conclude that an invertible positive $\mathcal{M}(2)$ admits a representing measure with at most $7$ atoms. \ While this would already represent a significant improvement on the upper bound given by Carath\' eodory's Theorem ($15$ atoms), we have been able to establish that all positive invertible $\mathcal{M}(2)$'s actually have \textit{flat} extensions, and therefore their representing measures can have exactly $6$ atoms.

\section{\label{Tool} A New Tool}

We begin this section with a result that will allow us to convert a given moment problem into a simpler, equivalent, moment problem. \ One of the consequences of this result is the equivalence of the real and complex moment problems, via the transformation $x:=$ Re$[z]$ and $y:=$ Im$[z]$; this equivalence has been exploited amply in the theory of truncated moment problems. \ For us, however, this simplification will allow us to assume that the submatrix $\mathcal{M}(1)$ is the identity matrix. \ 

We adapt the notation in \cite{tcmp6} to the real case. \ For $a,b,c,d,e,f\in \mathbb{R}$, $bf-ce \ne 0$, let $\Psi(x,y)\equiv \left(\Psi_1(x,y),\Psi_2(x,y)\right):=\left( a+bx + cy, d+ex+fy\right)$ ($x,y\in \mathbb{R}$). \ Given $\beta ^{(2n)}$, define $%
\tilde{\beta}^{\left( 2n\right) }$ by $\tilde{\beta}_{ij}:=L_{\beta }(\Psi_1^i \Psi_2^j)$ ($0\leq i+j\leq 2n$), where $L_{\beta }$ denotes the \textit{Riesz functional} associated with $\beta $. \ (For $p(x,y)\equiv \sum p_{ij} x^i y^j$, the Riesz functional is given by $L_{\beta}(p):=p(\beta)\equiv \sum p_{ij} \beta_{ij}$.) \ It is straightforward to verify that $
L_{\tilde{\beta}}(p)=L_{\beta }\left( p\circ \Psi \right) $ for every $p$ of degree at most $n$. 

\begin{proposition}
\label{lininv}\textup{(Invariance under degree-one transformations; \cite{tcmp6})} \ Let $%
\mathcal{M}(n)$ and $\tilde{\mathcal{M}}(n)$ be the moment matrices associated with $\beta$ and 
$\tilde{\beta}$, and let $J\hat{p}:=\widehat{p\circ\Psi}$. \ Then the following statements hold.

\begin{enumerate}
\item  \label{lininv(1)}$\tilde{\mathcal{M}}(n)=J^{\ast}\mathcal{M}(n)J$.

\item  \label{lininv(2)}$J$ is invertible.

\item  \label{lininv(3)}$\tilde{\mathcal{M}}(n)\geq0\Leftrightarrow \mathcal{M}(n)\geq0$.

\item  \label{lininv(4)}$\operatorname{rank}\tilde{\mathcal{M}}(n)=\operatorname{rank}\mathcal{M}(n)$.

\item  \label{lininv(6)}$\mathcal{M}\left( n\right) $ admits a flat extension if and
only if $\tilde{\mathcal{M}}\left( n\right) $ admits a flat extension.

\end{enumerate}
\end{proposition}

We are now ready to put $\mathcal{M}(2)$ in ``normalized form." \ Without loss of generality, we always assume that $\beta_{00}=1$. \ Let $d_i$ denote the leading principal minors of $\mathcal{M}(2)$; in particular,
\begin{eqnarray*}
d_2 &=&-\beta_{10}^2+\beta_{20}\\
d_3&=& -\beta_{02} \beta_{10}^2+2 \beta_{01} \beta_{10} \beta_{11}-\beta_{11}^2-\beta_{01}^2 \beta_{20}+\beta_{02} \beta_{20}.
\end{eqnarray*}
Consider now the degree-one transformation
$$
\Psi(x,y)\equiv \left( a+bx + cy, d+ex+fy\right),
$$
where $ a:=\frac{\beta_{01}\beta_{20}-\beta_{10} \beta_{11}}{\sqrt{d_2 d_3}}$, $b:=\frac{\beta_{11}-\beta_{01} \beta_{10}}{\sqrt{d_2 d_3}}$, $c:=
- \sqrt{\frac{d_2}{d_3}}$ , $d:=-\frac{\beta_{10}}{\sqrt {d_2} }$, $e:= \frac{1}{\sqrt {d_2} }$, and $f:=0$. \ Note that $bf-ce=- \sqrt{\frac{1}{d_3}}\neq0$. \ Using this transformation, and a straightforward calculation, we can prove that any positive definite moment matrix $\mathcal{M}(2)$ can be transformed into the moment matrix
$$
\begin{pmatrix}
1 & 0& 0 & 1 & 0 & 1 \\
0& 1& 0 & \tilde{\beta}_{30} & \tilde{\beta}_{21}& \tilde{\beta}_{12} \\
0 & 0 & 1 & \tilde{\beta}_{21} & \tilde{\beta}_{12} & \tilde{\beta}_{03} \\
1 & \tilde{\beta}_{30} & \tilde{\beta}_{21} & \tilde{\beta}_{40} & \tilde{\beta}_{31} & \tilde{\beta}_{22} \\
0 & \tilde{\beta}_{21} & \tilde{\beta}_{12} & \tilde{\beta}_{31} & \tilde{\beta}_{22} & \tilde{\beta}_{13} \\
1 & \tilde{\beta}_{12} & \tilde{\beta}_{03} & \tilde{\beta}_{22} & \tilde{\beta}_{13} & \tilde{\beta}_{04}
\end{pmatrix}.
$$
Thus, without loss of generality, we can always assume that $\mathcal{M}(1)$ is the identity matrix. \ We will now introduce a new tool in the study of moment matrices: the decomposition of an invertible $\mathcal{M}(2)$ as a sum of a moment matrix of rank $5$ and a rank-one moment matrix. \

Assume now that $\mathcal{M}(2)$ is invertible and that the submatrix $\mathcal{M}(1)$ is the identity matrix. For $u \in \mathbb{R}$ decompose $\mathcal{M}(2)$ as follows:
$$\mathcal{M}(2)=\begin{pmatrix}
1 - u &0 & 0 & 1 & 0 & 1 \\
0 & 1 & 0 & \beta_{30} & \beta_{21} & \beta_{12} \\
 0 &0 & 1 & \beta_{21} & \beta_{12} & \beta_{03} \\
 1& \beta_{30} & \beta_{21} & \beta_{40} & \beta_{31} & \beta_{22} \\
0 & \beta_{21} & \beta_{12} & \beta_{31} & \beta_{22} & \beta_{13} \\
 1 & \beta_{12} & \beta_{03} & \beta_{22} & \beta_{13} & \beta_{04}
\end{pmatrix}
+
\begin{pmatrix}
 u & 0&0&0&0&0 \\
0 & 0&0&0&0&0 \\
0 & 0&0&0&0&0 \\
0 & 0&0&0&0&0 \\
0 & 0&0&0&0&0 \\
0 & 0&0&0&0&0
\end{pmatrix}.
$$
Denote the first summand by $\widehat{\mathcal{M}(2)}$ and the second summand by $\mathcal{P}$. \ It is clear that $\mathcal{P}$ is positive semidefinite and has rank $1$ if and only if $u>0$, and in that case $\mathcal{P}$ is the moment matrix of the $1$-atomic measure $u \delta_{(0,0)}$, where $\delta_{(0,0)}$ is the point mass at the origin. 

\begin{proposition} \label{rankred} Let $\mathcal{M}(2)$, $\widehat{\mathcal{M}(2)}$ and $\mathcal{P}$ be as above, and let $u_0:=\frac{ \operatorname{det} \mathcal{M}(2)}{R_{11}}$, where $R_{11}$ is the $(1,1)$ entry in the positive matrix $R:=(\mathcal{M}(2))^{-1}$. \ Then, with this nonnegative value of $u$, we have (i) $\widehat{\mathcal{M}(2)}\ge 0$; \ (ii) $\operatorname{rank} \widehat{\mathcal{M}(2)}=5$; and (iii) $\widehat{\mathcal{M}(2)}$ is recursively generated. \ Moreover, $u_0$ is the only value of $u$ for which $\widehat{\mathcal{M}(2)}$ satisfies (i)--(iii).  
\end{proposition}

For the proof of Proposition \ref{rankred} we will need to following auxiliary result, which is an easy consequence of the multilinearity of the determinant.

\begin{lemma} \label{lem1} Let $M$ be an $n \times n$ invertible matrix of real numbers, let $E_{11}$ be the rank-one matrix with $(1,1)$-entry equal to $1$ and all other entries are equal to zero, and let $u \in \mathbb{R}$. \ Then $\operatorname{det} (M-u E_{11})=\operatorname{det} M-u \operatorname{det} M_{\{2,3,\cdots,n\}}$, where $M_{\{2,3,\cdots,n\}}$ denotes the $(n-1) \times (n-1)$ compression of $M$ to the last $n-1$ rows and columns. \ In particular, if $u=\frac{\operatorname{det} M}{(M^{-1})_{11}}$, then $\operatorname{det} (M-u E_{11})=0$.   
\end{lemma}

\begin{proof}[Proof of Proposition \ref{rankred}]
(ii) Observe that $6=\rank \mathcal{M}(2) \leq \rank \widehat{\mathcal{M}(2)} +\rank \mathcal{P} =\rank \widehat{\mathcal{M}(2)} +1$, so $\rank \widehat{\mathcal{M}(2)} \geq 5.$ \ Since $\det \widehat{\mathcal{M}(2)}=0$, we have $\rank \widehat{\mathcal{M}(2)}=5$. \newline

(i) Using the Nested Determinant Test starting at the lower right-hand corner of $\widehat{\mathcal{M}(2)}$, we know that $\widehat{\mathcal{M}(2)}$ is positive semidefinite since the nested determinants corresponding to principal minors of size $1$, $2$, $3$, $4$ and $5$ are all positive, and the rank of $\widehat{\mathcal{M}(2)}$ is $5$. \ This also implies that $1-u\geq0$. \ We now claim that $1-u$ is strictly positive. If  $1-u=0$, then the positive semidefiniteness of $\widehat{\mathcal{M}(2)}$ would force all entries in the first row to be zero. \ Since this is evidently false, we conclude that $1-u>0$. \newline

(iii) It is sufficient  to show that the first three columns of $\widehat{\mathcal{M}(2)}$ are linearly independent. \ Consider the third leading principal minor of $\widehat{\mathcal{M}(2)}$, which equals $1-u$, and is therefore positive. \ Thus, there is no linear dependence in this submatrix, and as a result the same holds in $\widehat{\mathcal{M}(2)}$.  

Finally, the uniqueness of $u_0$ as the only value satisfying (i)--(iii) is clear. \end{proof}

\section{\label{Proof} Proof of the Main Result} 

We first observe that by combining Proposition \ref{rankred} with Theorem \ref{quartic}, it suffices to consider the case when $\widehat{\mathcal{M}(2)}$ has a column relation corresponding to a pair of intersecting lines. \ For, in all other cases, there exists a representing measure for $\widehat{\mathcal{M}(2)}$ with exactly five atoms; when combined with the additional atom coming from the matrix $\mathcal{P}$, we see that $\mathcal{M}(2)$ admits a $6$-atomic representing measure.

We thus focus on the case when $\widehat{\mathcal{M}(2)}$ is subordinate to a degenerate hyperbola. \ After applying an additional degree-one transformation, we can assume, as in Proposition \ref{degenerate}, that the column relation $XY=0$ is present in $\widehat{\mathcal{M}(2)}$.  However, we may not continue to assume that the submatrix $\widehat{\mathcal{M}(1)}$ is the identity matrix, since the degree-one transformation that produces the column relation $XY=0$ will, in general, change the low-order moments. \ That is, $\widehat{\mathcal{M}(2)}$ is of the form
$$\widehat{\mathcal{M}(2)}=\left(
\begin{array}{cccccc}
 1 & a & b & c & 0 & d \\
 a & c & 0 & e & 0 & 0 \\
 b & 0 & d & 0 & 0 & f \\
 c & e & 0 & g & 0 & 0 \\
 0 & 0 & 0 & 0 & 0 & 0 \\
 d & 0 & f & 0 & 0 & h \\
\end{array}
\right).
$$
In this case, the original moment matrix $\mathcal{M}(2)$ is written as 
$$\mathcal{M}(2) = \widehat{\mathcal{M}(2)}+u \begin{pmatrix}
1 & p & q& p^2& p\, q & q^2
\end{pmatrix}^T \begin{pmatrix}
1 & p & q& p^2& p\, q & q^2
\end{pmatrix},$$
for some $u>0$ and $p\, q\neq 0$. \ That is, $\mathcal{M}(2)$ is the sum of a moment matrix of rank $5$ with column relation $XY=0$ and a positive scalar multiple of the moment matrix associated with the point mass at $(p,q)$, with $pq \ne 0$. \ Without loss of generality, we can assume that $p=q=1$ (this requires an obvious degree-one transformation, i.e., $\tilde{x}:=\frac{x}{p}, \; \tilde{y}:=\frac{y}{q}$). \ As a result, the form of $\mathcal{M}(2)$ is now as follows:
$$\mathcal{M}(2)=\left(
\begin{array}{cccccc}
 1+u & a+u & b+u & c+u & u & d+u \\
 a+u & c+u & u & e+u & u & u \\
 b+u & u & d+u & u & u & f+u \\
 c+u & e+u & u & g+u & u & u \\
 u & u & u & u & u & u \\
 d+u & u & f+u & u & u & h+u \\
\end{array}
\right)$$
We will show that $\mathcal{M}(2)$ admits a flat extension, and that will readily imply that it admits a $\rank \mathcal{M}(2)$-atomic (that is, 6-atomic) representing measure. \ The $B(3)$-block in an extension $\mathcal{M}(3)$ can be generated by letting $\beta_{41}=\beta_{32}=\beta_{23}=\beta_{14}=u$, so that $B(3)$ can thus be written as
$$ \left(
\begin{array}{cccc}
 e+u & u & u & f+u \\
 g+u & u & u & u \\
 u & u & u & h+u \\
 \beta_{50} & u & u & u \\
 u & u & u & u \\
 u & u & u & \beta_{05} \\
\end{array}
\right).$$
As usual, let $W:= \mathcal{M}(2)^{-1} B(3)$ and let $C(3)\equiv (C_{ij}):=W^{\ast} \mathcal{M}(2) W$. \ Note that if $C(3)$ turns out to be Hankel, then $\mathcal{M}(3)$ is a flat extension of $\mathcal{M}(2)$. \ Since $C(3)$ is symmetric, to ensure that $C(3)$ is Hankel (and therefore $\mathcal{M}(3)$ is a moment matrix) we only need to solve the  following system of equations:
\begin{eqnarray}\label{eq-E}
\begin{cases}
E_1:=C_{13}-C_{22} =0\\
E_2:= C_{14}-C_{23}=0\\
E_3:= C_{24}-C_{33}=0.
\end{cases}
\end{eqnarray}
This is a system of equations involving quadratic polynomials with $2$ unknown variables (the new moments $\beta_{50}$ and $\beta_{05}$). \ A straightforward calculation shows that $E_1=0$, $E_3=0$, and that
\begin{eqnarray*}
E_2 = 0~ \ \iff \!\!\!\!\!\!\!&&(c^2 - a e) (d^2 - b f)
 \beta_{50} \beta_{05}
+(c^2 - a e) (f^3 - 2 d f h + b h^2 - d^2 u + b f u)\beta_{50} \\
&&+(d^2 - b f)  (e^3 - 2 c e g + a g^2 - c^2 u + a e u) \beta_{05} \\
&&+ (e^3 - 2 c e g + a g^2 - c^2 u + a e u) (f^3 - 2 d f h + b h^2 -d^2 u + b f u)=0 \\
\ \iff \!\!\!\!\!\!\!&&\kappa \lambda \beta_{50} \beta_{05}+\kappa \mu \beta_{50}+\lambda \nu \beta_{05}+\nu \mu =0 ,
\end{eqnarray*}
where $\kappa$, $\lambda$, $\mu$ and $\nu$ have the obvious definitions. \ If $\kappa, \lambda \ne 0$, then $\beta_{05}=\frac{-\mu \nu+\kappa \mu \beta_{50}}{\kappa \lambda \beta_{50}+\lambda \nu}$ (for $\beta_{50} \ne -\frac{\nu}{\kappa}$), which readily implies that $E_2=0$ admits infinitely many solutions. \ When $\kappa = 0$ and $\lambda \ne 0$, we see that $E_2=\lambda \nu \beta_{05}+\mu \nu$, from which it follows that a solution always exists (and it is unique when $\nu \ne 0$). \ A similar argument shows that $\kappa \ne 0$ and $\lambda =0$ also yields a solution (which is unique when $\mu \ne 0$). \ We are thus left with the case when both $\kappa \equiv c^2-ae$ and $\lambda \equiv d^2-bf$ are equal to zero. \ Since $c$ and $d$ are in the diagonal of a positive semidefinite matrix, they must be positive. \ Thus, all of $a,$ $b$, $e$, and $f$ are nonzero and we can set $e:=c^2/a$ and $f:=d^2/b$. \ In this case, the moment matrix is
\begin{eqnarray*}
\mathcal{M}(2)=\left(
\begin{array}{cccccc}
 1+u & a+u & b+u & c+u & u & d+u \\
 a+u & c+u & u & \frac{c^2}{a}+u & u & u \\
 b+u & u & d+u & u & u & \frac{d^2}{b}+u \\
 c+u & \frac{c^2}{a}+u & u & g+u & u & u \\
 u & u & u & u & u & u \\
 d+u & u & \frac{d^2}{b}+u & u & u & h+u
\end{array}
\right)
\end{eqnarray*}
Let $k:=\det \mathcal{M}(2)/ \det \mathcal{M}(2)_{\{2,3,4,5,6\}}$. \ As in the proof of Proposition \ref{rankred}, we see that $k=\frac{-b^2 c-a^2 d+c d}{c d}>0$ and the first summand in the following decomposition of $\mathcal{M}(2)$ has rank $5$ and is positive semidefinite (note that the $(1,1)$-entry is $1+u-k$):
$$
\mathcal{M}(2)=\left(
\begin{array}{cccccc}
 \frac{b^2 c+a^2 d+c d u}{c d} & a+u & b+u & c+u & u & d+u \\
 a+u & c+u & u & \frac{c^2+a u}{a} & u & u \\
 b+u & u & d+u & u & u & \frac{d^2+b u}{b} \\
 c+u & \frac{c^2+a u}{a} & u & g+u & u & u \\
 u & u & u & u & u & u \\
 d+u & u & \frac{d^2+b u}{b} & u & u & h+u
\end{array}
\right)+
\left(
\begin{array}{cccccc}
 k & 0 & 0 & 0 & 0 & 0 \\
 0 & 0 & 0 & 0 & 0 & 0 \\
 0 & 0 & 0 & 0 & 0 & 0 \\
 0 & 0 & 0 & 0 & 0 & 0 \\
 0 & 0 & 0 & 0 & 0 & 0 \\
 0 & 0 & 0 & 0 & 0 & 0
\end{array}
\right)
$$
The only column relation in the first summand is
\begin{equation}\label{cr1}
XY=\frac{c d}{-b c-a d+c d}\textit{1} -\frac{a d}{-b c-a d+c d} X -\frac{b c }{-b c-a d+c d}Y=:\xi 1-\eta X - \theta Y.
\end{equation}
Unless $\eta \theta=-\xi$, the conic that represents this column relation is a nondegenerate hyperbola, and therefore the moment sequence associated to the moment matrix has a $5$-atomic measure, by Theorem \ref{quartic}. \ In the case when the conic in (\ref{cr1}) is a pair of intersecting lines (i.e., $(x+\theta)(y+\eta)=0$), we must have $c=a$ or $d=b$. \ 

Thus, the remaining two specific cases to cover are $\mathcal{M}(2)$ with $c=a$ or $d=b$. \ 
Since $\mathcal{M}(2)$ is invertible, for any $B(3)$ block we will be able to find $W$ such that $\mathcal{M}(2)W=B(3)$. \ We propose to use a $B(3)$ block with new moments $\beta_{32}=\beta_{23}=\beta_{14}=0$, and to then extend $\mathcal{M}(2)$ to $\mathcal{M}(3)$ using Smul'jan's Lemma, that is, we will define $C(3):=W^*B(3)$. \ The goal is to establish that $C(3)$ is a Hankel matrix, and that requires verification of (\ref{eq-E}). \ Before we begin our detailed analysis, we need to make a few observations. \ 

Let $d_i$ denote the principal minor of $\mathcal{M}(2)$ for $i=1,\ldots,6$; since $\mathcal{M}(2)$ is positive and invertible, we know that these minors are all positive. \ Then
\begin{eqnarray*}
d_5=-\frac{\left(b^2 c+a^2 d-c d\right) \left(-c^3+a^2 g\right) u}{a^2}  \text{ \ \ and \ \ }
d_6=\frac{ d_5\left(-d^3+b^2 h\right) }{ b^2},
\end{eqnarray*}
which implies
\begin{eqnarray} \label{minor}
\left(b^2 c+a^2 d-c d\right) \left(-c^3+a^2 g\right)<0 \qquad \text{ and }  \qquad -d^3+b^2 h >0 .
\end{eqnarray}
Next, we use \textit{Mathematica} to solve $E_1=0$ for $\beta_{50}$ and $E_3=0$ for $\beta_{05}$, and we obtain
\begin{eqnarray*}
&&\beta_{50}=\frac{1}{a^2 c \left(b^2 c+a^2 d-c d\right) \left(-d^3+b^2 h\right) u}
( \alpha_{11} \beta_{41}^2 + \alpha_{12} \beta_{41} +\alpha_{13} ), \\
&&\beta_{05}=\frac{1}{b^2 d \left(b^2 c+a^2 d-c d\right) \left(c^3-a^2 g\right)} (  \alpha_{21} \beta_{41} +\alpha_{22} ),
\end{eqnarray*}
where the $\alpha_{ij}$'s are polynomials in $a,b,c,d,g,h,$ and $u$.
Since $a,b\neq 0$, $c,d>0$, we can use (\ref{minor}) to show that both $\beta_{50}$ and $\beta_{05}$ above are well defined. \ We now substitute these values in $E_2$ and check that $E_2$ is a quadratic polynomial in $\beta_{41}$; indeed, we can readily show that the leading coefficient of $E_2$ is nonzero if $c=a$ or $d=b$. \ Thus, if the discriminant $\Delta $ of this quadratic polynomial is nonnegative, then (\ref{eq-E}) has at least one solution. \ We are now ready to deal with the two special cases: $c=a$ and $d=b$. \ If $c=a$, then
\begin{eqnarray*}
\Delta=\frac{a^2 u^2 (a-g)^2 \left(-d^3+b^2 h\right)^2 F_1(a,b,d,h) }{b^4 d^2},
\end{eqnarray*}
where $$F_1(a,b,d,h)= (-1+a)^2 b^2 h^2+
2 b^2 d \left(2 b^2-3 d+3 a d\right) h
-d^4 \left(3 b^2-4 d+4 a d\right)$$ is a concave upward quadratic polynomial in $h$. \ Notice that $\Delta \geq 0$ if and only if $F_1 \geq 0$, which means that the discriminant of $F_1$,  $\Delta_1:=
16 b^2 d^2 \left(b^2-d+a d\right)^3$, needs to be zero or negative. \ In this case, we observe that
\begin{eqnarray*}
&&c=a>0,\\
&&d_3= -a b^2+a d-a^2 d+a u-a^2 u-b^2 u+d u-a d u > 0,\\
&&d_4= -d_3 (a-g)>0 \; (\Rightarrow a-g<0),\\
&&d_5= a \left(b^2-d+a d\right) (a-g) u >0,
\end{eqnarray*}
which leads to $b^2-d+a d<0$. \ Therefore, $\Delta_1 <0$ and $\Delta >0$.

Similarly, if $d=b$, then
\begin{eqnarray*}
\Delta=\frac{\left(-c^3+a^2 g\right)^2 (b-h)^2  u^2 F_2(a,b,c,h)}{a^4},
\end{eqnarray*}
where 
\begin{eqnarray*}
F_2(a,b,d,h)&=&(-1+b)^2 h^2 c^2 +
2 a (-1+b) \left(-2 b^3+3 b^2 h+a h^2-b h^2\right) c \\
&&+ a^2 \left(-4 a b^3+b^4+6 a b^2 h-2 b^3 h+a^2 h^2-2 a b h^2+b^2 h^2\right)
\end{eqnarray*}
is a concave upward quadratic polynomial in $c$. \ The discriminant of $F_2$ is $\Delta_2:=16 a^2 (-1+b)^2 b^3 (b-h)^3$; we observe that $d=b>0$ and $d_6=-d_5(b-h)>0$, which leads to $b-h<0$. Therefore, $\Delta_2<0$ and $\Delta >0$, which completes the proof. \qed

\bigskip
\textit{Acknowledgments}. \ The authors are deeply grateful to the referee for many suggestions that led to significant improvements in the presentation. \ Many of the examples, and portions of the proofs of
some results in this paper were obtained using calculations with the
software tool \textit{Mathematica \cite{Wol}}.

\end{document}